\documentclass[12pt,oneside]{amsart}

\usepackage{amssymb,amsmath,mathtools,fourier-otf}
\usepackage{upgreek}
\usepackage{aliascnt}
\usepackage[inline]{enumitem}
\usepackage{verbatim}
\usepackage[a4paper,left=25mm,right=25mm,top=20mm,bottom=20mm,footskip=10mm,headheight=14pt]{geometry}
\usepackage{fancyhdr}
\usepackage[
 backend=biber,
 style=numeric,
 sorting=nyt,
 sortcites=true,
 giveninits=true,
 maxbibnames=99,
 doi=true,
 url=false,
 isbn=false
]{biblatex}
\usepackage{xurl}
\usepackage[breaklinks,pdfstartview=FitH,pdfusetitle]{hyperref}
\hypersetup{
 colorlinks=true,
 linkcolor=black,
 citecolor=black,
 urlcolor=black,
}

\newtheorem{theorem}{Theorem}[section]
\newaliascnt{corollary}{theorem}
\newtheorem{corollary}[corollary]{Corollary}
\aliascntresetthe{corollary}
\newaliascnt{lemma}{theorem}
\newtheorem{lemma}[lemma]{Lemma}
\aliascntresetthe{lemma}
\newaliascnt{proposition}{theorem}
\newtheorem{proposition}[proposition]{Proposition}
\aliascntresetthe{proposition}
\newaliascnt{claim}{theorem}

\aliascntresetthe{claim}
\theoremstyle{definition}
\newaliascnt{definition}{theorem}
\newtheorem{definition}[definition]{Definition}
\aliascntresetthe{definition}
\newaliascnt{remark}{theorem}
\newtheorem{remark}[remark]{Remark}
\aliascntresetthe{remark}

\DeclareMathOperator{\diam}{diam}
\newcommand{\Ntree}{\mathsf N}
\newcommand{\T}{\mathsf T}
\newcommand{\Tprime}{\mathsf T'}
\newcommand{\rootv}{\mathsf r}
\newcommand{\parent}{\mathsf p}
\newcommand{\pb}{\mathord\Diamond}
\newcommand{\rep}{\mathrm{x}}
\newcommand{\bdry}{\partial}
\newcommand{\N}{\mathbb N}
\newcommand{\eps}{\varepsilon}

\title{Regular Ultrametric Skeletons}

\author[M. Mendel]{Manor Mendel}
\address{Department of Mathematics and Computer Science\\
The Open University of Israel\\
1 University Road\\
Raanana 43107, Israel}
\email{manorme@openu.ac.il}

\date{}

\begin{document}

\begin{abstract}
The ultrametric skeleton theorem extracts
from every compact metric probability space a subset of
ultrametric distortion $O(1/\varepsilon)$ that carries a measure whose balls
are controlled by the $(1-\varepsilon)$-power of the original measure on dilated concentric balls.
We prove a two-sided version for arbitrary
compact metric spaces:
for every ball centered on the skeleton, the
skeleton measure also has a lower bound in terms of the original measure on
a smaller nonconcentric ball contained in it.
We also give a short proof of
the original skeleton theorem and improve the dilation of its
control balls from $\exp(O(1/\varepsilon^2))$ to
$O(1/\varepsilon)$.
\end{abstract}

\maketitle

\section{Introduction}

Fix a metric space $(X,d)$, a point $x\in X$, and a radius
$r\in[0,\infty)$.  The corresponding closed and open balls are denoted,
respectively, by
\[
 B_d(x,r)=\{z\in X:d(x,z)\le r\},
 \qquad
 B_d^\circ(x,r)=\{z\in X:d(x,z)<r\}.
\]
An ultrametric space is a metric
space $(U,\rho)$ satisfying the strengthened triangle inequality
\[
 \rho(x,y)\le\max\{\rho(x,z),\rho(y,z)\}
 \qquad(x,y,z\in U).
\]
We say that $(S,d)$ has ultrametric distortion at most $D$ if there is an
ultrametric $\rho$ on $S$ such that
\[
 d(x,y)\le \rho(x,y)\le D\cdot d(x,y), \qquad \forall x,y\in S.
\]
Ultrametrics are the metric incarnation of rooted hierarchies.
Many questions that are difficult in an arbitrary metric space become tractable on an (approximately) ultrametric space.

The ultrametric skeleton theorem~\cite{MN-skeleton,MN-hausdorff} states that for every compact metric
space $(X,d)$, every Borel probability measure $\mu$ on $X$, and every
$\eps\in(0,1)$, there exist a compact subset $S\subseteq X$ of ultrametric
distortion $O(1/\eps)$ and a Borel probability measure $\nu$ supported on
$S$ such that, for every $x\in X$ and $r\ge0$,
\begin{equation}
 \nu(B_d(x,r))\le
 \mu(B_d(x,C_\eps r))^{1-\eps}.
\label{eq:classical-skeleton-intro}
\end{equation}
Here $C_\eps$ depends only on $\eps$.
The argument in
\cite[Theorem~1.1 and the discussion following
Theorem~1.2]{MN-skeleton} gives
$C_\eps=\exp(O(\eps^{-2}))$.
For doubling spaces, a later argument~\cite{mendel2021simple} gives linear dilation \(C_\varepsilon=O(1/\varepsilon)\), at the expense of a doubling-dependent factor in the measure estimate.

In rough terms, the estimate prevents \(\nu\) from concentrating on balls to which \(\mu\) assigns little mass.
It therefore certifies largeness of $S$ simultaneously at all locations and scales.
For example, the theorem yields the sharp nonlinear Dvoretzky statement that
every compact metric space has a closed subset of Hausdorff dimension at
least a $(1-\eps)$-fraction of the original dimension and ultrametric
distortion $O(1/\eps)$ \cite{MN-hausdorff,MN-skeleton}.
It has applications to metric Ramsey theory,
online algorithms, proximity data structures, generic chaining, and
geometric measure theory; see
\cite{BLMN05,BBM,MN07,Tal05,KMZ,MZ,MN-hausdorff,MN-skeleton} and the references therein.
A two-sided version was subsequently
obtained for doubling spaces in~\cite{mendel2021dvoretzkytype}, where it
was used to prove an asymptotically sharp Dvoretzky-type theorem for Ahlfors-regular
spaces.

The results of this paper are as follows.
\begin{enumerate}[label=(\Alph*),leftmargin=2.2em]
\item\label{intro:simple}
We give a direct and short proof of the original ultrametric
skeleton theorem for arbitrary compact metric spaces.
\item\label{intro:dilation}
This proof yields a linear dilation $C_\eps=O(1/\eps)$ without a
doubling assumption.
\item\label{intro:regular}
We prove a regular version of the theorem that also gives a lower bound for the skeleton measure in terms of the \(\mu\)-mass of a smaller ball contained in the original ball.
\end{enumerate}

The proofs of \ref{intro:simple}--\ref{intro:dilation} apply the repeated
Bartal decomposition directly to subsets of $X$.  For
\ref{intro:regular}, we retain the net-tree and partial-boundary machinery
from~\cite{mendel2021dvoretzkytype}, which supplies the interior balls
needed for the lower estimate.  The core--envelope bookkeeping introduced
here replaces the use of the doubling hypothesis in that argument and
provides the probability-flow estimates used in both the original and the regular arguments.

\begin{theorem}[Regular ultrametric skeleton]
\label{thm:main}
There are universal constants $c_1,c_2>0$ such that the following holds.
For every compact metric probability space $(X,d,\mu)$,
and every $t\in\{2,3,\ldots\}$,
there exist a compact subset $S\subseteq X$, an ultrametric $\rho$ on $S$,
and a Borel probability measure $\nu$ supported on $S$ such that
\begin{enumerate}[label=(\roman*)]
\item\label{item:distortion}
 $d(x,y)\le\rho(x,y)\le32t\,d(x,y)$ for all $x,y\in S$;
\item\label{item:upper}
 for every $x\in X$ and $r\ge0$,
 \begin{equation}
 \label{eq:upper-main}
 \nu(B_d(x,r))\le \mu(B_d(x,c_1tr))^{1-1/t};
 \end{equation}
\item\label{item:lower}
 for every $y\in S$ and $r>0$, there is $z\in X$ such that
 \begin{equation}
 \label{eq:lower-main}
 \begin{gathered}
 B_d\left(z,{r}/{(c_2t)}\right)\subseteq B_d(y,r),\quad \text{and} \quad
 \nu(B_d(y,r))\ge
 \frac12\mu\bigl(B_d(z,{r}/{(c_2t)})\bigr)^{1-1/t}.
 \end{gathered}
 \end{equation}
\end{enumerate}
\end{theorem}

Without optimizing the constants, they may be chosen so that
$c_1\leq 129$ and $c_2\leq 5120$.

Taking \(t=\lceil1/\eps\rceil\), we have
\(1-1/t\geq1-\eps\) and \(t\leq2/\eps\).
Since \(\mu\) is a probability measure,
items \ref{item:distortion}--\ref{item:upper} recover
the original theorem with both distortion and dilation \(O(1/\eps)\).
Item~\ref{item:lower} also
controls the lower mass of balls in \(S\).  For example, if
\[
 a r^Q\leq\mu(B_d(x,r))\leq b r^Q
\]
for some constants $a,b,Q>0$ and over all $x\in X$ and $r\in(0,\diam(X)]$, then the two estimates
imply
\[
 \nu(B_d(y,r))\asymp_{a,b,Q,t} r^{Q(1-1/t)}
 \qquad(y\in S,\ r\in(0,\diam(S)]).
\]
Thus, in this case, \((S,d)\) is Ahlfors \(Q(1-1/t)\)-regular and has ultrametric distortion $O(t)$.

The paper is organized as follows.  Section~\ref{sec:decompositions} recalls
Bartal's decomposition and develops the core--envelope variant used here.
In Section~\ref{sec:warmup}, we provide a short proof of the original
ultrametric skeleton
theorem of~\cite{MN-skeleton} with linear dilation.
Section~\ref{sec:main-proof} adds the net-tree construction
from~\cite{mendel2021dvoretzkytype} needed
for the lower estimate and proves \autoref{thm:main}.

\subsection*{Acknowledgments}
The impetus to revisit this problem stemmed from an invitation
to present the papers~\cite{mendel2021simple,mendel2021dvoretzkytype}
at the Fields Institute.
The author developed a new
idea and presented it to GPT-5.6 Sol.
GPT developed a complete proof and prepared an initial draft.
The author then revised the draft and substantially simplified the proofs
with the assistance of GPT.
The arguments in the proofs were independently checked by the author.

\section{Decompositions}
\label{sec:decompositions}

We begin with the decomposition used to construct ultrametric subsets of
metric measure spaces.  For nonempty sets $A,B\subseteq X$, write
\[
 d(A,B):=\inf\{d(a,b):a\in A,\ b\in B\}.
\]

For a compact metric space $(Z,d)$ with a finite Borel measure $m$, a Borel
set $A\subseteq Z$, and $\Delta\ge0$, write
\begin{equation}
\label{eq:localized-mass}
 m^\Delta(A):=\sup_{x\in A}m(B_d(x,\Delta/4)\cap A),
\end{equation}
with $m^\Delta(\varnothing)=0$.
We will repeatedly use the following immediate monotonicity property: if
$A\subseteq C\subseteq Z$ are Borel and $0\le\delta\le\Delta$, then
\begin{equation}
\label{eq:localized-monotonicity}
 m^\delta(A)\le m^\Delta(C).
\end{equation}
Also, if $A$ is compact, $m(A)>0$, and $\delta>0$, then
$m^\delta(A)>0$.  Indeed, finitely many balls of radius $\delta/4$,
centered in $A$, cover $A$, and at least one has positive $m$-measure in
$A$.

The following is a parameterized compact variant of Bartal's decomposition lemma~\cite{bartal2021advances} that was proved
in \cite[Lemma~2.2]{mendel2021dvoretzkytype}.

\begin{lemma}[Bartal's decomposition]
\label{lem:Bartal}
Let $(Z,d)$ be compact and let $m$ be a finite Borel measure on $Z$ with
$m(Z)>0$.  Suppose
\[
 0<\Delta<2\diam(Z), \qquad t\in\{2,3,\ldots\}.
\]
There are nonempty disjoint compact sets
$P,Q\subseteq Z$, with $m(P)>0$, such that, on writing
$Q^c=Z\setminus Q$,
\begin{align}
 d(P,Q)&\ge \frac{\Delta}{8t},
 &\diam(Q^c)&\le\frac \Delta2,
 &\diam(P)&\le\Bigl(\frac12-\frac1{4t}\Bigr)\Delta,
 \label{eq:Bartal-geometry}
\end{align}
and
\begin{equation}
\label{eq:Bartal-sharp}
 m(P)\ge m(Q^c)
 \Bigl(\frac{m^{\Delta/2}(Q^c)}{m^\Delta(Z)}\Bigr)^{1/t},
 \text{ and in particular, }\frac{m(P)}{m^{\Delta/2}(P)^{1/t}} \geq \frac{m(Q^c)}{m^\Delta(Z)^{1/t}}.
\end{equation}
\end{lemma}

\begin{lemma}[Repeated decomposition]
\label{lem:repeated-decomp}
Let $(X,d)$ be a compact metric space, let $m$ be a finite Borel measure
on $X$, let $Z\subseteq X$ be compact with $m(Z)>0$, and let
$t\in\{2,3,\ldots\}$.  If $\Delta\ge\diam(Z)$, then there is a finite
set $I$ together with
pairwise disjoint Borel sets $(E_i)_{i\in I}$ contained in $Z$ and compact sets
$P_i\subseteq E_i$ with $m(P_i)>0$ satisfying:
\begin{enumerate}[label=(\alph*)]
\item
\(
 \sum_{i\in I}m(E_i)=m(Z)
\);
\item\label{repeat:local}
 $m(E_i)\le m^\Delta(Z)$ for every $i\in I$;
\item\label{repeat:diameter}
 $\diam(E_i)\le\Delta/4$ for every $i\in I$.  Moreover, all indices
 except possibly one satisfy
 \[
  \diam(P_i)\le\Bigl(\frac14-\frac1{8t}\Bigr)\Delta;
 \]
 the possible exceptional index, called the \emph{final residual},
 satisfies $P_i=E_i$;
\item\label{repeat:separation}
 $d(P_i,P_j)\ge \Delta/(16t)$ whenever $i,j\in I$, $i\neq j$;
\item\label{repeat:charge}
For every $i\in I$,
\begin{equation}
\label{eq:repeat-strong}
 \frac{m(P_i)}
 {\bigl(m^{\Delta/4}(P_i)\bigr)^{1/t}}
 \ge
 \frac{m(E_i)}{\bigl(m^\Delta(Z)\bigr)^{1/t}}.
\end{equation}
\end{enumerate}
\end{lemma}

\begin{proof}
If $\Delta=0$, then $Z$ is a singleton.  Taking $I=\{0\}$ and
$P_0=E_0=Z$ proves all the assertions.  Assume henceforth that $\Delta>0$.

We iteratively peel off subsets.
Start with
$R_0=Z$.
As long as
$m(R_{i-1})>0$ and $\diam(R_{i-1})>\Delta/4$, apply
\autoref{lem:Bartal} to $R_{i-1}$, with the restriction of $m$ to
$R_{i-1}$ and with parameter $\Delta/2$.  Denote its output by
$P_i,Q_i$, put
\[
E_i=R_{i-1}\setminus Q_i,\qquad R_i=Q_i.
\]
This process is finite.  Otherwise, choosing a point from each $P_i$
would give infinitely many points that are pairwise separated by
$\Delta/(16t)$, contradicting the compactness of $Z$.
Notice also that the stopping condition makes every
application legitimate, since
\[
 \frac{\Delta}2<2\diam(R_{i-1}).
\]

Let $f\ge0$ be the number of iterations, and let $R_f$ be the final
residual.  If $m(R_f)>0$, record the additional pair
$P_{f+1}=E_{f+1}=R_f$ and set $I=\{1,\ldots,f+1\}$.  Otherwise omit it
and set $I=\{1,\ldots,f\}$.

Every $P_i$ is compact and has positive measure, and every $E_i$ is Borel.
Furthermore, the sets $(E_i)_{i\in I}$ are pairwise disjoint.  The complement
of their union in $Z$ is empty when the final residual is recorded and equals
the null set $R_f$ otherwise.  Hence their measures sum to $m(Z)$.
\autoref{repeat:diameter} follows from \eqref{eq:Bartal-geometry} for
$i\le f$ and from the stopping condition for the final residual
$i=f+1$, when it is present.
Since $P_i\subseteq E_i$, every $E_i$ is nonempty.  Choose $z_i\in E_i$.
Then
\[
 E_i\subseteq B_d(z_i,{\Delta}/4)\cap Z,
\text{\quad and therefore\quad}
 m(E_i)\le m(B_d(z_i,{\Delta}/4)\cap Z)
 \le m^\Delta(Z).
\]
This proves \autoref{repeat:local}.

For every $1\leq i<j\leq |I|$, we have
$P_j\subseteq R_{j-1}\subseteq R_i=Q_i$.  Therefore
\[
 d(P_i,P_j)\ge d(P_i,Q_i)\ge\frac{\Delta}{16t}.
\]

It remains to prove the charge inequality.  For $i\leq f$,
\eqref{eq:Bartal-sharp} and \eqref{eq:localized-monotonicity} yield
\[
 \frac{m(P_i)}
 {\bigl(m^{\Delta/4}(P_i)\bigr)^{1/t}}
 \ge
 \frac{m(E_i)}
 {\bigl(m^{\Delta/2}(R_{i-1})\bigr)^{1/t}}
 \ge
 \frac{m(E_i)}
 {\bigl(m^\Delta(Z)\bigr)^{1/t}}.
\]
Finally, if $f+1\in I$, then
$E_{f+1}=P_{f+1}\subseteq Z$, so
\eqref{eq:localized-monotonicity} gives
\[
 \frac{m(P_{f+1})}
 {\bigl(m^{\Delta/4}(P_{f+1})\bigr)^{1/t}}
 \ge
 \frac{m(E_{f+1})}{\bigl(m^\Delta(Z)\bigr)^{1/t}}.
\]
This proves \autoref{repeat:charge}.
\end{proof}

\section{The original ultrametric skeleton theorem}
\label{sec:warmup}

Most applications of the ultrametric skeleton theorem use only its original
form from~\cite{MN-skeleton}, without a lower bound on the skeleton
measure.  We give a direct and short proof of this one-sided theorem
using the decomposition from the preceding section.

\begin{theorem}
\label{thm:upper-skeleton}
Let $(X,d,\mu)$ be a compact metric probability space.
For every $t\in\{2,3,\ldots\}$ there exist a compact set
$S\subseteq X$, an ultrametric $\rho$ on $S$, and a Borel probability
measure $\nu$ supported on $S$ such that
\begin{enumerate}[label=(\roman*)]
\item\label{upper-skeleton:distortion}
 $d(x,y)\le\rho(x,y)\le16t\,d(x,y)$ for all $x,y\in S$;
\item\label{upper-skeleton:measure}
 for every $x\in X$ and $r\ge0$,
 \begin{equation}
 \label{eq:upper-skeleton}
  \nu(B_d(x,r))
  \le \mu\bigl(B_d(x,33tr)\bigr)^{1-1/t}.
 \end{equation}
\end{enumerate}
\end{theorem}

\begin{definition}[Tree boundary]
\label{def:tree-boundary}
Let $\T$ be a locally finite rooted tree without leaves, meaning that
every vertex has finitely many children and at least one child.  Its
\emph{boundary} $\bdry\T$ is the set of infinite branches starting at
$\rootv_\T$.  For $u\in\T$, the corresponding cylinder is
\[
 \bdry u=\{b\in\bdry\T:u\in b\}.
\]
Suppose that $\delta:\T\to[0,\infty)$ is nonincreasing along branches,
tends to zero along every branch, and is positive at every vertex having
at least two children.  For distinct $a,b\in\bdry\T$, let $a\wedge b\in \T$
denote their last common vertex and set
\begin{equation}
\label{eq:tree-boundary-ultrametric}
 \rho(a,b)=\delta(a\wedge b),
 \qquad \rho(a,a)=0.
\end{equation}
Then $\rho$ is an ultrametric, its positive-radius balls are
cylinders, and $(\bdry\T,\rho)$ is compact.  Indeed, finite branching and
König's lemma imply that the labels tend to zero uniformly across each
level, so the boundary is totally bounded; a Cauchy sequence of branches
stabilizes on every finite initial segment and hence converges.
These assertions are standard; see, for example,
\cite[Lemma~7]{mendel2021simple}.
\end{definition}

\begin{definition}[Probability flow]
\label{def:probability-flow}
A \emph{probability flow} on $\T$ is a map
$\sigma:\T\to[0,\infty)$ such that
\[
 \sigma(\rootv_\T)=1,
 \qquad
 \sigma(u)=\sum_{v\in\parent_\T^{-1}(u)}\sigma(v)
 \quad(u\in\T).
\]
Every probability flow determines a unique Borel probability measure
$\overline\nu$ on $\bdry\T$ satisfying
$\overline\nu(\bdry u)=\sigma(u)$ for every $u\in\T$.  Indeed, the flow
defines a finitely additive set function on finite unions of cylinders.
This set function is a premeasure: if such a union is written as a
countable disjoint union of sets of the same form, compactness and the fact
that cylinders are clopen imply that only finitely many terms are nonempty.
The conclusion follows from the Carath\'eodory extension theorem
\cite[Theorem~1.53]{Klenke}.
\end{definition}

\begin{proof}[Proof of \autoref{thm:upper-skeleton}]
If $\diam(X)=0$, then $X$ is a singleton, and the conclusion is immediate.  Henceforth assume that $\diam(X)>0$.
We recursively construct a locally finite rooted tree
$\T$ without leaves, compact cores $C_u\subseteq X$ of positive measure,
Borel envelopes $C_u\subseteq E_u\subseteq X$, and labels
$\delta(u)>0$.  At the root put
\[
 C_{\rootv_\T}=E_{\rootv_\T}=X,
 \qquad
 \delta(\rootv_\T)=\diam(X).
\]
Suppose inductively that $\diam(C_u)\le\delta(u)$, and apply
\autoref{lem:repeated-decomp} to $C_u$ with parameter
$\Delta=\delta(u)$.  For every resulting pair $(P_v,E_v)$, create a child
$v$ of $u$, put $C_v=P_v$, and assign the label
\[
 \delta(v)=\frac{\delta(u)}4.
\]
The output of \autoref{lem:repeated-decomp} gives for $v\neq w$, $v,w\in \parent_\T^{-1}(u)$,
\begin{align}
 C_v\subseteq E_v&\subseteq C_u,
 &\diam(E_v)&\le\delta(v),                                      \label{eq:warmup-scales}\\
 d(C_v,C_w)&\ge\frac{\delta(u)}{16t},                  \label{eq:warmup-separation}\\
 \sum_{v\in\parent_\T^{-1}(u)}\mu(E_v)&=\mu(C_u),
 &\mu(E_v)&\le\mu^{\delta(u)}(C_u).                            \label{eq:warmup-envelopes}
\end{align}
Moreover, \eqref{eq:repeat-strong} becomes
\begin{equation}
\label{eq:warmup-charge}
 \frac{\mu(E_v)}
 {\bigl(\mu^{\delta(u)}(C_u)\bigr)^{1/t}}
 \le
 \frac{\mu(C_v)}
 {\bigl(\mu^{\delta(v)}(C_v)\bigr)^{1/t}}.
\end{equation}

Define $\sigma$ recursively by
\begin{equation}\label{eq:prob-flow}
 \sigma(\rootv_\T)=1,
 \qquad
 \sigma(v)=\sigma(u)\frac{\mu(E_v)}{\mu(C_u)}
 \quad(v\in\parent_\T^{-1}(u)).
\end{equation}
All denominators in this definition are positive.  Moreover,
\[
 \sum_{v\in\parent_\T^{-1}(u)}\sigma(v)
 =\frac{\sigma(u)}{\mu(C_u)}
   \sum_{v\in\parent_\T^{-1}(u)}\mu(E_v)
 =\sigma(u).
\]
Thus $\sigma$ is a probability flow.
We claim that, for every $u\in\T$,
\begin{equation}
\label{eq:warmup-flow-bound}
 \sigma(u)\le
 \frac{\mu(C_u)}
 {\bigl(\mu^{\delta(u)}(C_u)\bigr)^{1/t}},
 \qquad
 \sigma(u)\le\mu(E_u)^{1-1/t}.
\end{equation}
At the root, the second inequality is an equality, while the first follows
from
\[
 \mu^{\delta(\rootv_\T)}(C_{\rootv_\T})
 \le \mu(C_{\rootv_\T})=1.
\]
If the inequalities hold at $u$, then
\begin{align*}
 \sigma(v)
 &=\sigma(u)\frac{\mu(E_v)}{\mu(C_u)}
\le
 \frac{\mu(E_v)}
 {\bigl(\mu^{\delta(u)}(C_u)\bigr)^{1/t}}
 \le
 \frac{\mu(C_v)}
 {\bigl(\mu^{\delta(v)}(C_v)\bigr)^{1/t}}.
\end{align*}
Here the first inequality is the inductive hypothesis and the second is
\eqref{eq:warmup-charge}.
The middle expression is at most
$\mu(E_v)^{1-1/t}$ by \eqref{eq:warmup-envelopes}.  This proves
\eqref{eq:warmup-flow-bound} by induction.

The labels decrease by exactly a factor of four.  Along every branch the
cores are nested nonempty compact sets, and \eqref{eq:warmup-scales} shows
that their diameters tend to zero.  Their intersection therefore consists
of a single point; denote it by $\pi(b)$ for a branch $b\in\bdry\T$.  If
$a,b\in\bdry\T$ are distinct, $u=a\wedge b$, and $v,w$ are the children of
$u$ on the two branches, then their endpoints belong to $C_v$ and $C_w$,
respectively.  Hence \eqref{eq:warmup-scales} and
\eqref{eq:warmup-separation} give
\[
 \frac{\delta(u)}{16t}
 \le d(\pi(a),\pi(b))\le\delta(u).
\]
In particular, $\pi$ is injective and is $1$-Lipschitz from the compact
boundary ultrametric with label map $\delta$.  Thus
$S=\pi(\bdry\T)$ is compact.  Transfer the boundary ultrametric to $S$.
The preceding inequalities give
\[
 d(\pi(a),\pi(b))
 \le\delta(u)
 =\rho(\pi(a),\pi(b))
 \le16t\,d(\pi(a),\pi(b)).
\]
This proves \autoref{upper-skeleton:distortion}.

Let $\overline\nu$ be the boundary measure induced by $\sigma$, and put
$\nu=\pi_\#\overline\nu$ the pushforward measure on $S$.  Then
\[
 \nu(\pi(\bdry u))=\sigma(u)
 \qquad(u\in\T).
\]
Fix $x\in X$ and $r>0$.  If $B_d(x,r)\cap S$ is nonempty, choose
$y\in B_d(x,r)\cap S$.  Let $b_y=\pi^{-1}(y)$ be the unique branch
corresponding to $y$, and let $u$ be the first vertex of $b_y$ for which
$\delta(u)\le32tr$.  Such a vertex exists because the labels tend to zero.
We claim that
\[
 B_d(x,r)\cap S\subseteq\pi(\bdry u).
\]
This is immediate if $u=\rootv_\T$.  Otherwise, suppose that
$z\in S\setminus\pi(\bdry u)$, and let $b_z=\pi^{-1}(z)$.  Then
$b_y\wedge b_z$ is a strict ancestor of $u$, and the minimality of $u$
gives
\[
 \delta(b_y\wedge b_z)
 \ge\delta(\parent_\T(u))>32tr.
\]
The distortion estimate already proved above yields
\[
 d(y,z)\ge\frac{\rho(y,z)}{16t}
 =\frac{\delta(b_y\wedge b_z)}{16t}>2r.
\]
Thus $z$ cannot also belong to $B_d(x,r)$, proving the claim.
Since $y\in C_u\subseteq E_u$ and
$\diam(E_u)\le\delta(u)$ (by \eqref{eq:warmup-scales} when $u$ is
nonroot and by definition at the root),
\[
 E_u\subseteq B_d(y,\delta(u))
 \subseteq B_d(x,(32t+1)r)
 \subseteq B_d(x,33tr).
\]
Consequently \eqref{eq:warmup-flow-bound} gives
\[
 \nu(B_d(x,r))
 \le\sigma(u)
 \le\mu(E_u)^{1-1/t}
 \le\mu(B_d(x,33tr))^{1-1/t}.
\]
The case $B_d(x,r)\cap S=\varnothing$ is immediate, and the case $r=0$
follows by continuity from above.  This proves
\autoref{upper-skeleton:measure}.
\end{proof}

\begin{remark}
The proof of \autoref{thm:upper-skeleton} follows the same general scheme as
the proof of the one-sided ultrametric skeleton theorem for doubling spaces
in~\cite{mendel2021simple}: a hierarchy obtained from Bartal
decompositions defines the ultrametric, and a probability flow controls its
cylinders.  In the latter proof the flow is constructed from the
subadditive quantity
\[
 \frac{\mu(C_u)}
 {\bigl(\mu^{\diam(C_u)}(C_u)\bigr)^{1/t}},
\]
and the doubling hypothesis is used to compare its denominator with
$\mu(C_u)$.  Here the core--envelope pairs supplied by
\autoref{lem:repeated-decomp} replace that comparison: the envelopes
partition $C_u$ in measure and hence define the flow directly through
\eqref{eq:prob-flow}; the charge inequality~\eqref{eq:repeat-strong} and
\autoref{repeat:local} of \autoref{lem:repeated-decomp} then yield both estimates in
\eqref{eq:warmup-flow-bound}.  In particular, the estimate involving
$E_u$ is what removes any dependence on a doubling constant.
\end{remark}

\section{Regular ultrametric skeleton theorem}
\label{sec:main-proof}

Here we prove the regular ultrametric skeleton theorem, \autoref{thm:main}.
The proof adapts that of \autoref{thm:upper-skeleton} and uses the
net-tree machinery developed in~\cite{mendel2021dvoretzkytype} to obtain
a lower bound on the skeleton measure.
We next recall the relevant terminology from
\cite[Definitions~3.1, 3.3, and~3.5]{mendel2021dvoretzkytype}.

\begin{definition}[Net-tree]
\label{def:net-tree}
Let $\kappa\ge1$.  A $\kappa$-net-tree over a compact metric space
$(X,d)$ is a locally finite rooted tree $\Ntree$ without leaves, together
with a label $\Delta_{\Ntree}:\Ntree\to[0,\infty)$ and a representative
map $\rep:\Ntree\to X$.  Here $\rootv_\Ntree$ denotes the root and
$\parent_\Ntree:\Ntree\setminus\{\rootv_\Ntree\}\to\Ntree$ maps a vertex
to its parent.  Write
$\operatorname{Desc}_{\Ntree}(a)\subseteq\Ntree$ for the set consisting of
$a$ and all its descendants.  The following axioms are required.
\begin{enumerate}[label=(\roman*)]
\item Labels are nonincreasing down the tree and tend to zero along
 every branch.
\item\label{defnet:covering}
 For every vertex $a$,
 \[
 \{\rep(c):c\in\operatorname{Desc}_{\Ntree}(a)\}
 \subseteq B_d\bigl(\rep(a),\Delta_{\Ntree}(a)\bigr).
 \]
\item\label{defnet:packing}
 If $a$ and $b$ are unrelated (neither is an ancestor of the other) and
 $\Delta_{\Ntree}(b)\le\Delta_{\Ntree}(a)$, then
 \[
 B_d\left(\rep(a),
   \frac{\Delta_{\Ntree}(\parent_{\Ntree}(a))}{\kappa}\right)
 \cap\{\rep(c):c\in\operatorname{Desc}_{\Ntree}(b)\}=\varnothing.
 \]
\end{enumerate}
The representative sequence on every branch converges: this follows
immediately from the covering and vanishing-label axioms (see also
\cite[Proposition~3.4]{mendel2021dvoretzkytype}).  The image of the net-tree is
the set of all such branch limits, and the net-tree is \emph{surjective}
if its image is $X$.

For a vertex $a$ and $0<s\le\Delta_{\Ntree}(a)$, its first-crossing cut
at scale $s$ is
\begin{equation} \label{eq:first-crossing-cut}
 \operatorname{Cut}_{\Ntree}(a,s)=
 \begin{cases}
 \{a\},&s=\Delta_{\Ntree}(a),\\
 \{b\in\operatorname{Desc}_{\Ntree}(a):
   \Delta_{\Ntree}(b)\le s<
   \Delta_{\Ntree}(\parent_{\Ntree}(b))\},&s<\Delta_{\Ntree}(a).
 \end{cases}
\end{equation}
\end{definition}

We use the following net-tree constructed in \cite[Section~3]{mendel2021dvoretzkytype}.

\begin{proposition}
\label{prop:net-cells}
Every nonempty compact metric space $(X,d)$ admits a surjective $20$-net-tree
$\Ntree$.  For every $a\in\Ntree$ there is a Borel set
$\pb(a)\subseteq X$, called its partial boundary, with the following
properties.
\begin{enumerate}[label=(\alph*)]
\item\label{net:partition}
 At the root, $\pb(\rootv_{\Ntree})=X$.  Moreover, for
 every vertex $a$ and every $0<s\le\Delta_{\Ntree}(a)$,
 \[
 \pb(a)=\bigsqcup_{b\in\operatorname{Cut}_{\Ntree}(a,s)}\pb(b),
 \]
 where $\sqcup$ denotes a disjoint union.
\item\label{net:covering}
 Every $x\in\pb(a)$ satisfies
 $d(x,\rep(a))\le\Delta_{\Ntree}(a)$.
\item\label{net:packing}
 If $a$ is not the root, then
 \begin{equation} \label{eq:nettree-packing}
 B_d^\circ\left(\rep(a),
       {\Delta_{\Ntree}(\parent_{\Ntree}(a))}/{20}\right)
 \subseteq\pb(a).
 \end{equation}
\item\label{net:finite}
 Every first-crossing cut $\operatorname{Cut}_{\Ntree}(a,s)$ is finite.
\item\label{net:levels}
 If $\diam(X)>0$, then
 $\Delta_{\Ntree}(\rootv_{\Ntree})=\diam(X)$ and every child has one quarter
 of its parent's label.  Thus the labels are
 $\diam(X)4^{-k}$, $k\in\N\cup\{0\}$.
\end{enumerate}
\end{proposition}

Items \ref{net:partition}--\ref{net:packing} follow from
\cite[Propositions~3.6 and~3.7]{mendel2021dvoretzkytype}; the refinement
identity in \ref{net:partition} is obtained by iterating the child
partition there.  Item~\ref{net:finite} follows from local finiteness and
the vanishing of labels, as noted immediately after
\cite[Definition~3.3]{mendel2021dvoretzkytype}.  Item~\ref{net:levels}
comes from the finite-net construction in the proof of
\cite[Proposition~3.8]{mendel2021dvoretzkytype}.  These are all the
net-tree facts used below.  Indeed, if $a'$ descends from $a$, then
\autoref{net:partition} of \autoref{prop:net-cells} and
\autoref{defnet:covering} of \autoref{def:net-tree} give, respectively,
\begin{equation}
\label{eq:rep-projection}
 \pb(a')\subseteq\pb(a),
 \qquad d(\rep(a'),\rep(a))\le\Delta_{\Ntree}(a).
\end{equation}

The last ingredient is a normalized version of the balanced-subtree
lemma from~\cite{mendel2021dvoretzkytype}.
We include the short proof for the special case needed here; see
also \cite[Lemma~4.3]{mendel2021dvoretzkytype}.

\begin{lemma}[Balanced probability flow]
\label{lem:balanced}
Let $\T$ be a locally finite rooted tree without leaves, and let
$\zeta:\T\to(0,\infty)$ satisfy
\begin{align}
 \zeta(\rootv_\T)&=1,
 &\zeta(u)&\le\sum_{v\in\parent_\T^{-1}(u)}\zeta(v),
 &\zeta(v)&\le\zeta(u) \quad\text{if $v$ descends from $u$}.
 \label{eq:zeta-assumptions}
\end{align}
Then there exist an ancestor-closed rooted subtree without leaves
$\Tprime\subseteq\T$, with the parent relation inherited from $\T$, and
a probability flow $\sigma$ on $\Tprime$ such that
\begin{equation}
\label{eq:balanced-conclusion}
 \sigma(\rootv_\T)=1,
 \qquad
 \frac12\zeta(u)<\sigma(u)\le\zeta(u),
 \qquad
 \sigma(u)=\sum_{v\in\parent_\T^{-1}(u)\cap\Tprime}\sigma(v).
\end{equation}
\end{lemma}

\begin{proof}
Set $\sigma(\rootv_\T)=1$.  Suppose $u$ has been retained and $\sigma(u)$
has been defined with
\begin{equation} \label{eq:balanced-inductive-hyp}
\zeta(u)/2<\sigma(u)\le\zeta(u).
\end{equation}
Choose an
inclusion-minimal set $L\subseteq\parent_\T^{-1}(u)$ for which
$\sum_{v\in L}\zeta(v)\ge\sigma(u)$; it exists by~\eqref{eq:balanced-inductive-hyp} and the subadditivity in
\eqref{eq:zeta-assumptions}.
If $L=\{v\}$, retain only $v$ and put $\sigma(v)=\sigma(u)$.  The defining
property of $L$ gives $\sigma(v)\le\zeta(v)$, while~\eqref{eq:balanced-inductive-hyp} and the monotonicity in
\eqref{eq:zeta-assumptions} give
\[\sigma(v)=\sigma(u)>\zeta(u)/2\geq\zeta(v)/2.\]
If $|L|\ge2$, then every
$v\in L$ satisfies $\zeta(v)<\sigma(u)$, while inclusion-minimality gives
\[\sum_{w\in L\setminus\{v\}}\zeta(w)<\sigma(u).\]
Hence \(\sum_{v\in L}\zeta(v)<2\sigma(u)\). In this case, retain precisely the children of \(u\) that belong to \(L\), and set
\[
 \sigma(v)=
 \frac{\sigma(u)\zeta(v)}{\sum_{w\in L}\zeta(w)}
 \qquad(v\in L).
\]
The desired inequalities and additivity follow.  The set $L$ is nonempty,
so every retained vertex has a retained child.  Iterating over the levels
therefore constructs a subtree $\Tprime$ without leaves, together with
$\sigma$.
\end{proof}

\medskip

We now lift \autoref{lem:repeated-decomp} through the net-tree of
\autoref{prop:net-cells}.

\begin{lemma}[Hierarchy]
\label{lem:hierarchy}
Let $(X,d)$ be compact with $\diam(X)>0$, let $\mu$ be a Borel probability
measure on $X$, and let $t\in\{2,3,\ldots\}$.  There is a universal constant
$c_2\le 5120$ and a $c_2t$-net-tree $\T$ over $X$ (not necessarily
surjective) with label map
$\Delta:\T\to(0,\infty)$.  Moreover, there are positive-measure Borel
``core--envelope'' pairs $(C_u,E_u)_{u\in\T}$ with the following properties.
\begin{enumerate}[label=(\Alph*)]
\item\label{hier:root}
 $C_{\rootv_\T}=E_{\rootv_\T}=X$ and $\Delta(\rootv_\T)=\diam(X)$.
\item\label{hier:nesting}
 For every child $v$ of $u$,
 $C_v\subseteq E_v\subseteq C_u$.  The child envelopes are pairwise
 disjoint, and
 \[
 \sum_{v\in\parent_\T^{-1}(u)}\mu(E_v)=\mu(C_u).
 \]
\item\label{hier:diameter}
 $\diam(C_u)\le \Delta(u)$, $\diam(E_u)\le2\Delta(u)$, and
 $\Delta(v)\le \Delta(u)/2$ for every child $v$ of $u$.
\item\label{hier:separation}
 Distinct children $v,w$ of $u$ satisfy
 \[
 d(C_v,C_w)\ge\frac{\Delta(u)}{32t}.
 \]
\item\label{hier:interior}
 Every nonroot vertex $v$ satisfies
 \begin{equation}
 \label{eq:hierarchy-interior}
 B_d\Bigl(\rep(v),\frac{\Delta(\parent_\T(v))}{c_2t}\Bigr)\subseteq C_v.
 \end{equation}
\item\label{hier:flow}
 There is a probability flow $\sigma$ on an ancestor-closed
 rooted subtree without leaves $\Tprime\subseteq\T$, with the inherited
 parent relation, such that
 \begin{equation}
 \label{eq:hierarchy-flow}
 \frac12\mu(C_u)^{1-1/t}<\sigma(u)\le\mu(E_u)^{1-1/t}
 \qquad(u\in\Tprime).
 \end{equation}
\end{enumerate}
\end{lemma}

\begin{proof}
We take $c_2=5120$ and construct $\T$ recursively from a surjective
$20$-net-tree $\Ntree$ and partial boundaries as in
\autoref{prop:net-cells}.  Equip the countable set $\Ntree$ with the
possibly infinite atomic measure determined by
\[
 \widetilde\mu(\{a\})=\mu(\pb(a)).
\quad\text{Thus,}\quad
 \widetilde\mu(A)=\sum_{a\in A}\mu(\pb(a))
\]
for every finite $A\subseteq\Ntree$.  We use this measure only on the
finite antichains appearing in the construction.  When $A$ is such a
same-level antichain, its partial-boundary cells are disjoint and
\begin{equation}
\label{eq:atomic-cell-mass}
 \widetilde\mu(A)=\mu\left(\bigcup_{a\in A}\pb(a)\right).
\end{equation}
A vertex $u\in\T$ is represented by two finite
antichains $\widetilde C_u,\widetilde Z_u\subseteq\Ntree$.
The set $\widetilde C_u$ consists of vertices from a single level of $\Ntree$.
The set \(\widetilde Z_u\) consists of descendants of vertices in \(\widetilde C_u\).
All vertices in \(\widetilde Z_u\) lie on a single level of \(\Ntree\), whose label is no larger than that of \(\widetilde C_u\).
We also maintain that
$\widetilde C_u\subseteq\widetilde Z_{\parent_\T(u)}$, for $u\neq \rootv_\T$.
We write
\[
 C_u=\bigcup_{a\in\widetilde C_u}\pb(a).
\]
At the root $\rootv_\T$ of $\T$, set
$\widetilde C_{\rootv_\T}=\{\rootv_{\Ntree}\}$ and use
$\rep(\rootv_{\Ntree})$ as its representative.  Set
$\Delta(\rootv_\T)=\diam(X)$ and $E_{\rootv_\T}=X$.  Null partial-boundary
cells may be discarded below the root.
\autoref{net:partition} of \autoref{prop:net-cells} gives
$C_{\rootv_\T}=X$.

Suppose inductively that $u\in\T$, $\widetilde C_u\subseteq\Ntree$, and
$\Delta(u)$ have been defined, that $\widetilde C_u$ is a finite antichain
contained in one $\Delta_\Ntree$ level, and that
$\diam(C_u)\le\Delta(u)$.
If $u$ is a nonroot vertex we denote its parent $w=\parent_\T(u)$ and
suppose also that
$\widetilde C_u\subseteq\widetilde Z_w$ and
$\Delta(u)\le\Delta(w)/2$.
Put
\begin{equation}
\label{eq:refinement-scale}
 s_u=\max\left\{\diam(X)4^{-k}:k\in\N\cup\{0\},\
 \diam(X)4^{-k}\le\frac{\Delta(u)}{256t}\right\},
 \qquad
 \frac{\Delta(u)}{1024t}<s_u\le\frac{\Delta(u)}{256t}.
\end{equation}
The maximum exists because the labels of $\Ntree$ are
$\diam(X)4^{-k}$.  At the root,
$s_{\rootv_\T}<\Delta_{\Ntree}(\rootv_{\Ntree})$.  If $u$ is a child of
$w$, then every $a\in\widetilde C_u\subseteq\widetilde Z_w$ has net-tree
label $s_w$.  To see that $s_u\le s_w$, suppose otherwise.  Since both
scales are four-adic, $s_u\ge4s_w>\Delta(w)/(256t)$, contradicting
$s_u\le\Delta(u)/(256t)\le\Delta(w)/(512t)$.
Thus the cuts in the uniform definition
\[
 \widetilde Z_u=
 \bigsqcup_{a\in\widetilde C_u}
 \operatorname{Cut}_{\Ntree}(a,s_u)
\]
are legitimate at every vertex.  When
$s_u=\Delta_{\Ntree}(a)$, equality is covered by the convention
\[
 \operatorname{Cut}_{\Ntree}(a,\Delta_{\Ntree}(a))=\{a\}.
\]
The definition of a first-crossing cut in
\eqref{eq:first-crossing-cut}, \autoref{net:finite} of
\autoref{prop:net-cells}, and the four-adic labels show that
$\widetilde Z_u$ is a finite antichain on one net-tree level and that
every $a\in\widetilde Z_u$ satisfies
\begin{equation}
\label{eq:coarse-parent-scale}
 \Delta_{\Ntree}(a)=s_u,
 \qquad
 \Delta_{\Ntree}(\parent_{\Ntree}(a))=4s_u.
\end{equation}
In particular, its vertices are unrelated and nonroot.
By \autoref{defnet:packing} of \autoref{def:net-tree}, their representatives
are distinct.  Thus \autoref{net:packing} of \autoref{prop:net-cells} gives
$\rep(a)\in\pb(a)\subseteq C_u$ for $a\in\widetilde Z_u$.  For
$a,b\in\Ntree$, write
\(
 \tilde d(a,b)=d(\rep(a),\rep(b)).
\)
All localized masses on the
finite antichains below are computed with respect to $\tilde d$.  Then
\[
 \diam_{\tilde d}(\widetilde Z_u)
 =\diam\{\rep(a):a\in\widetilde Z_u\}
 \le\diam(C_u)\le \Delta(u).
\]
Discard zero-mass atoms from $\widetilde Z_u$ and retain the same notation
for the remaining set.  Before this deletion, the cells indexed by
$\widetilde Z_u$ partition $C_u$ because $\widetilde Z_u$ refines
$\widetilde C_u$.  Only finitely many null cells are deleted, so their
union is $\mu$-null.
Thus \autoref{net:partition} of
\autoref{prop:net-cells} gives
\begin{equation}
\label{eq:discrete-total-mass}
 \widetilde\mu(\widetilde Z_u)=\mu(C_u).
\end{equation}
Apply \autoref{lem:repeated-decomp} to the finite metric space
$(\widetilde Z_u,\tilde d)$, with the restriction of the atomic measure
$\widetilde\mu$ and parameter $\Delta(u)$.
For every resulting core--envelope pair
$(\widetilde P_v,\widetilde E_v)$ of
\autoref{lem:repeated-decomp}, create a child $v$ of $u$ in $\T$ and define
\begin{equation}
\label{eq:child-core-envelope}
 \widetilde C_v=\widetilde P_v,
 \quad
 C_v=\bigcup_{a\in\widetilde P_v}\pb(a),
 \quad
 E_v=\bigcup_{a\in\widetilde E_v}\pb(a).
\end{equation}
Recall that among the core--envelope pairs of
\autoref{lem:repeated-decomp} there may be a distinguished pair, called
the final-residual pair.
Assign the label
\begin{equation}
\label{eq:child-labels}
 \Delta(v)=
 \begin{cases}
   \Delta(u)/4+2s_u,&\text{if $v$ is the final-residual child},\\
  \Delta(u)/4-8s_u,&\text{otherwise}.\\
\end{cases}
\end{equation}
By \eqref{eq:refinement-scale},
$0<\Delta(v)\le\Delta(u)/2$, and
$\widetilde C_v=\widetilde P_v\subseteq\widetilde Z_u$.  Thus these parts
of the induction hypothesis are preserved.
For later use, \eqref{eq:refinement-scale} and $t\ge2$ give
\[
10s_u\le\frac{\Delta(u)}{8t},\qquad
18s_u\le\frac{\Delta(u)}4,\qquad
2s_u\le\frac{\Delta(u)}{32t},\qquad
 \frac{s_u}{5}>\frac{\Delta(u)}{5120t}.
\]
Since
$\widetilde P_v\subseteq\widetilde E_v\subseteq\widetilde Z_u$, the
partial-boundary partition and \eqref{eq:atomic-cell-mass} give
\[
 C_v\subseteq E_v\subseteq C_u,\qquad
 \mu(C_v)=\widetilde\mu(\widetilde P_v),\qquad
 \mu(E_v)=\widetilde\mu(\widetilde E_v).
\]
The sets $(\widetilde E_v)_v$ are pairwise disjoint and their masses sum to
$\widetilde\mu(\widetilde Z_u)=\mu(C_u)$.  Hence the child envelopes are
pairwise disjoint,
and
\[
 \sum_{v\in\parent_\T^{-1}(u)}\mu(E_v)=\mu(C_u).
\]
This proves \autoref{hier:nesting}.  It also shows that every vertex has at
least one child and every core has positive measure.

We verify the geometric assertions.  Every point in a partial-boundary
cell represented at scale $s_u$ is within $s_u$ of its representative.
By \autoref{repeat:diameter} of \autoref{lem:repeated-decomp},
$\diam_{\tilde d}(\widetilde E_v)\le\Delta(u)/4$.
For a non-final-residual child $v$, \autoref{repeat:diameter} of
\autoref{lem:repeated-decomp}
and \eqref{eq:refinement-scale} give
\[
 \diam(C_v)
 \le\frac{\Delta(u)}4-\frac{\Delta(u)}{8t}+2s_u
 \le\frac{\Delta(u)}4-8s_u=\Delta(v).
\]
For the final-residual child $v$, $\widetilde P_v=\widetilde E_v$, and hence
$\diam(C_v)\le\Delta(u)/4+2s_u=\Delta(v)$.  In either case,
\[
 \diam(E_v)\le\frac{\Delta(u)}4+2s_u\le2\Delta(v),
 \qquad 0<\Delta(v)\le\frac{\Delta(u)}2.
\]

By \autoref{repeat:separation} of \autoref{lem:repeated-decomp},
$\tilde d(\widetilde C_v,\widetilde C_w)\ge \Delta(u)/(16t)$.
Passing to their partial-boundary cells loses at most
$2s_u$, so
\[
 d(C_v,C_w)\ge\frac{\Delta(u)}{16t}-2s_u
 \ge\frac{\Delta(u)}{32t}.
\]
Choose any $a\in\widetilde C_v$ and use $\rep(v):=\rep(a)$ as the representative
of $v$ in $\T$.  By
\eqref{eq:coarse-parent-scale} and the packing property of partial
boundaries~\eqref{eq:nettree-packing},
\[
 B_d\Bigl(\rep(v),\frac{\Delta(u)}{5120t}\Bigr)
 \subseteq
 B_d^\circ\left(\rep(v),\frac{s_u}{5}\right)
 =
 B_d^\circ\Bigl(\rep(a),
   \frac{\Delta_{\Ntree}(\parent_{\Ntree}(a))}{20}\Bigr)
 \stackrel{\eqref{eq:nettree-packing}}\subseteq\pb(a)\subseteq C_v.
\]
This proves \eqref{eq:hierarchy-interior} for $c_2=5120$.

These facts also verify that the hierarchy is a $c_2t$-net-tree.
Indeed, every vertex has finitely many children and at least one child,
and the positive labels decrease by a factor of at least two, hence tend
to zero along every branch.  If $w$ descends from $u$, then
$\rep(w)\in C_w\subseteq C_u$ and $\rep(u)\in C_u$, so
\[
 d(\rep(w),\rep(u))\le\diam(C_u)\le\Delta(u),
\]
which is the covering axiom.  If $u$ and $v$ are unrelated, their cores
lie in distinct child cores of their least common ancestor and are
therefore disjoint by \autoref{hier:separation}.  Every representative
below $v$ lies in $C_v$, whereas \autoref{hier:interior} gives
\[
 B_d\Bigl(\rep(u),
   \frac{\Delta(\parent_\T(u))}{c_2 t}\Bigr)\subseteq C_u
 \qquad(u\ne\rootv_\T).
\]
Thus the packing axiom holds (in fact, without needing its label-order
hypothesis).

It remains to establish the flow assertion, \autoref{hier:flow}.
Since $\widetilde Z_u$ is finite and has positive total mass,
\[
 0<\widetilde\mu^{\Delta(u)}(\widetilde Z_u)\le\mu(C_u).
\]

We first compare localized masses at consecutive resolutions.  Let $v$ be
a child of $u$ in $\T$.  Since $\widetilde Z_v$ refines
$\widetilde C_v=\widetilde P_v$, let
\(
 \operatorname{pr}_v:\widetilde Z_v\rightarrow\widetilde P_v
\)
send each vertex to its unique ancestor in $\widetilde P_v$; when no
refinement occurs, this is the identity.  By
\eqref{eq:rep-projection} and \eqref{eq:coarse-parent-scale},
\[
 \tilde d\bigl(b,\operatorname{pr}_v(b)\bigr)\le s_u
 \qquad(b\in\widetilde Z_v),
\]
and the cut partition gives
\[
 \sum_{\substack{b\in\widetilde Z_v\\
                  \operatorname{pr}_v(b)=a}}
 \widetilde\mu(\{b\})
 =\widetilde\mu(\{a\})
 \qquad(a\in\widetilde P_v).
\]
Consequently, for every $\delta\ge0$,
\begin{equation}
\label{eq:localized-comparison}
 \widetilde\mu^\delta(\widetilde Z_v)
 \le
 \widetilde\mu^{\delta+8s_u}(\widetilde P_v).
\end{equation}
Indeed, if $b_0\in\widetilde Z_v$ and
\[
 A=\{b\in\widetilde Z_v:
       \tilde d(b,b_0)\le\delta/4\},
\]
then
\[
 \widetilde\mu(A)
 \le\widetilde\mu(\operatorname{pr}_v(A)),
\]
and $\operatorname{pr}_v(A)$ is contained in the ball in
$\widetilde P_v$ centered at $\operatorname{pr}_v(b_0)$ with radius
\[
 \frac{\delta}{4}+2s_u
 =\frac{\delta+8s_u}{4}.
\]
Taking the supremum over $b_0$ proves
\eqref{eq:localized-comparison}.

For a non-final-residual child, \eqref{eq:child-labels} gives
\[
 \Delta(v)+8s_u=\frac{\Delta(u)}4,
\]
and hence
\begin{equation} \label{eq:tilde-Zv-leq-tilde-Pv}
 \widetilde\mu^{\Delta(v)}(\widetilde Z_v)
 \le\widetilde\mu^{\Delta(u)/4}(\widetilde P_v).
\end{equation}
For the final-residual child,
\[
 \Delta(v)+8s_u=\frac{\Delta(u)}4+10s_u<\Delta(u).
\]
Here the strict inequality follows from \eqref{eq:refinement-scale} and
$t\ge2$.
Since $\widetilde P_v\subseteq\widetilde Z_u$,
\eqref{eq:localized-comparison} followed by
\eqref{eq:localized-monotonicity} gives
\begin{equation} \label{eq:tilde-Zv-leq-tilde-Zu}
 \widetilde\mu^{\Delta(v)}(\widetilde Z_v)
 \le \widetilde\mu^{\Delta(v)+8s_u}(\widetilde P_v)
 \le \widetilde\mu^{\Delta(u)}(\widetilde Z_u).
\end{equation}

It follows that every child $v$ of $u$ satisfies
\begin{equation}
\label{eq:child-capacity}
 \frac{\mu(E_v)}
 {\widetilde\mu^{\Delta(u)}(\widetilde Z_u)^{1/t}}
 \le
 \frac{\mu(C_v)}
 {\widetilde\mu^{\Delta(v)}(\widetilde Z_v)^{1/t}}.
\end{equation}
Indeed, for a non-final-residual child this follows from
\eqref{eq:repeat-strong} and~\eqref{eq:tilde-Zv-leq-tilde-Pv}.
For the final-residual child it follows from $E_v=C_v$ and~\eqref{eq:tilde-Zv-leq-tilde-Zu}.
Moreover,
\autoref{repeat:local} of \autoref{lem:repeated-decomp} gives
$\mu(E_v)\le\widetilde\mu^{\Delta(u)}(\widetilde Z_u)$, and hence
\begin{equation}
\label{eq:child-envelope-capacity}
 \frac{\mu(E_v)}
 {\widetilde\mu^{\Delta(u)}(\widetilde Z_u)^{1/t}}
 \le\mu(E_v)^{1-1/t}.
\end{equation}

Define
\[
 \eta(u)=
 \min\left\{
 \frac{\mu(C_u)}
 {\widetilde\mu^{\Delta(u)}(\widetilde Z_u)^{1/t}},
 \mu(E_u)^{1-1/t}
 \right\}.
\]
Equations \eqref{eq:child-capacity} and
\eqref{eq:child-envelope-capacity} show that
\[
 \frac{\mu(E_v)}
 {\widetilde\mu^{\Delta(u)}(\widetilde Z_u)^{1/t}}
 \le\eta(v)
\]
for every child $v$ of $u$.  Therefore
\[
 \eta(u)
 \le\frac{\mu(C_u)}
 {\widetilde\mu^{\Delta(u)}(\widetilde Z_u)^{1/t}}
 =\sum_{v\in\parent_\T^{-1}(u)}
   \frac{\mu(E_v)}
   {\widetilde\mu^{\Delta(u)}(\widetilde Z_u)^{1/t}}
 \le\sum_{v\in\parent_\T^{-1}(u)}\eta(v).
\]
Thus $\eta$ is subadditive.  Since
$0<\widetilde\mu^{\Delta(u)}(\widetilde Z_u)
\le\mu(C_u)\le\mu(E_u)$,
\begin{equation}
\label{eq:eta-sandwich}
 \mu(C_u)^{1-1/t}
 \le\eta(u)
 \le\mu(E_u)^{1-1/t}.
\end{equation}
Furthermore, if $v$ is a child of $u$, then
\[
 \eta(v)
 \le\mu(E_v)^{1-1/t}
 \le\mu(C_u)^{1-1/t}
 \le\eta(u).
\]
Therefore $\eta$ is nonincreasing along branches.
At the root,
\eqref{eq:eta-sandwich} gives
$\eta(\rootv_\T)=1$.
Applying \autoref{lem:balanced} to
$\eta$ yields an ancestor-closed rooted subtree without leaves
$\Tprime\subseteq\T$ and a probability flow $\sigma$ satisfying
\[
 \frac12\mu(C_u)^{1-1/t}
 <\sigma(u)
 \le\mu(E_u)^{1-1/t}
 \qquad(u\in\Tprime).
\]
This proves \autoref{hier:flow}.
\end{proof}

\medskip

\begin{proof}[Proof of \autoref{thm:main}]
We prove the theorem with $c_1=129$ and $c_2=5120$.
If $\diam(X)=0$, take $S=X$, $\nu=\mu$, and the singleton ultrametric.  Assume
henceforth that $\diam(X)>0$ and apply \autoref{lem:hierarchy}.  We use
only the pruned tree $\Tprime$ and retain the notation
$C_u,E_u,\Delta(u),\sigma(u)$ on it.
Since labels decrease by a factor at least two, every infinite branch
$b=(u_0,u_1,\ldots)$ has $\Delta(u_n)\to0$.  The compact sets
$\overline{C_{u_n}}$ are nested, nonempty, and have diameters tending to
zero.  Their intersection is a singleton; denote it by $\pi(b)$.
If $a,b\in\bdry\Tprime$ are distinct and $u=a\wedge b$, then
their endpoints lie in $\overline{C_u}$ and in the closures of two
distinct child cores of $u$.  Since taking closures does not change the
distance between two sets,
\autoref{hier:diameter} and \autoref{hier:separation} of
\autoref{lem:hierarchy} give
\[
 d(\pi(a),\pi(b))\le \Delta(u)
 \le32t\,d(\pi(a),\pi(b)).
\]
In particular, $\pi$ is injective and $1$-Lipschitz from the compact
ultrametric boundary in \autoref{def:tree-boundary}, with label map
$\Delta$.  Hence $S=\pi(\bdry\Tprime)$ is compact.  Transfer the boundary
ultrametric to $S$ by setting
\[
 \rho(\pi(a),\pi(b))=
 \begin{cases}
  \Delta(a\wedge b),&a\ne b,\\
  0,&a=b.
 \end{cases}
\]
The preceding inequalities show that
$d(x,y)\le\rho(x,y)\le32t\,d(x,y)$ for $x,y\in S$.
This proves \autoref{item:distortion}.

Let $\overline\nu$ be the boundary measure induced by the probability
flow $\sigma$ as in \autoref{def:probability-flow}, and set
$\nu=\pi_\#\overline\nu$, the pushforward measure on $S$.
Extend $\nu$ to a measure on $X$ by
$\nu(B)=\nu(B\cap S)$ for every Borel set $B\subseteq X$.
Since $\pi$ is continuous and injective, $\pi(\bdry u)$ is compact and
$\pi^{-1}(\pi(\bdry u))=\bdry u$; hence
\begin{equation}
\label{eq:cylinder-mass}
 \nu(\pi(\bdry u))=\sigma(u).
\end{equation}

Fix $x\in X$ and $r>0$.  If $B_d(x,r)\cap S=\varnothing$, the upper
estimate is immediate.  Otherwise choose $y\in B_d(x,r)\cap S$.  For
every $z\in B_d(x,r)\cap S$,
\[
 \rho(y,z)\le32t\,d(y,z)\le64tr.
\]
Let $b_y=\pi^{-1}(y)$ be the unique branch corresponding to $y$, and let
$v$ be its first vertex whose label is at most $64tr$.  Then
\[
 B_d(x,r)\cap S\subseteq\pi(\bdry v),
\qquad \Delta(v)\le64tr.
\]
Indeed, if $v$ is not the root, the least common ancestor of $b_y$ and any
branch not passing through $v$ has label at least
$\Delta(\parent_\T(v))>64tr$.  If $v$ is the root, the containment is
immediate.
Moreover, $y\in\overline{C_v}\subseteq\overline{E_v}$, and taking closures
does not change diameter.  Hence
\[
 E_v\subseteq\overline{E_v}\subseteq B_d(y,2\Delta(v))
 \subseteq B_d(x,(128t+1)r)
 \subseteq B_d(x,c_1tr).
\]
Using \eqref{eq:hierarchy-flow} and \eqref{eq:cylinder-mass},
\[
 \nu(B_d(x,r))\le\sigma(v)
 \le\mu(E_v)^{1-1/t}
 \le\mu(B_d(x,c_1tr))^{1-1/t}.
\]
For $r_n\downarrow0$, the closed balls $B_d(x,r_n)$ and
$B_d(x,c_1tr_n)$ decrease to $\{x\}$.  Thus the case $r=0$ follows by
continuity from above of both finite measures.  This proves
\autoref{item:upper}.

Finally, fix $y\in S$ and $r>0$, let $b_y=\pi^{-1}(y)$, and let $v$ be
the first vertex of $b_y$ with $\Delta(v)\le r$.  Since
$y\in\overline{C_v}$ and
$\diam(\overline{C_v})=\diam(C_v)\le r$, we have
$\overline{C_v}\subseteq B_d(y,r)$. Hence,
\begin{equation}
\label{eq:lower-cylinder-contained}
 \pi(\bdry v)\subseteq \overline{C_v} \subseteq B_d(y,r).
\end{equation}
If $v$ is the root, then $r\ge\diam(X)$ and $B_d(y,r)=X$. Thus, the required lower bound follows.  Otherwise $\Delta(\parent_\T(v))>r$, so \autoref{hier:interior} of
\autoref{lem:hierarchy} gives
\[
 B_d\Bigl(\rep(v),\frac{r}{c_2t}\Bigr)
 \subseteq
 B_d\Bigl(\rep(v),\frac{\Delta(\parent_\T(v))}{c_2t}\Bigr)
 \subseteq C_v.
\]
By \eqref{eq:cylinder-mass} and
\eqref{eq:lower-cylinder-contained}, $\nu(B_d(y,r))\ge\sigma(v)$.  Using
\eqref{eq:hierarchy-flow} and the preceding ball inclusion now yields
\begin{align*}
 \nu(B_d(y,r))
 &\ge\sigma(v)>\frac12\mu(C_v)^{1-1/t}\ge\frac12
 \mu\Bigl(B_d\Bigl(\rep(v),\frac{r}{c_2t}\Bigr)\Bigr)^{1-1/t}.
\end{align*}
The ball above is also contained in
$C_v\subseteq\overline{C_v}\subseteq B_d(y,r)$.  Taking $z=\rep(v)$
proves \autoref{item:lower}.
\end{proof}

\begin{remark}
\label{rem:zero-radius}
The preceding lower-bound argument for balls of positive radius does not
by itself imply
\(
 \nu(\{y\})\ge\tfrac12\mu(\{y\})^{1-1/t}
\)
for an atom $y$.  The same boundary issue is implicit
in~\cite{mendel2021dvoretzkytype}.
\end{remark}

\begin{remark}
As noted in~\cite[Section 6]{mendel2021dvoretzkytype}, it is not known whether the
lower bound in \autoref{thm:main} can be strengthened to use concentric
balls.  More precisely, it is not known whether
\[
\nu(B_d(y,r))\ge
 \frac12\mu\bigl(B_d(y,{r}/{(c_2t)})\bigr)^{1-1/t},
\]
for every $y\in S$ and $r>0$.
If true, this strengthening would, by continuity from above, also imply the
corresponding estimate for $r=0$ and hence resolve the issue in
\autoref{rem:zero-radius}.
\end{remark}

In terms of $\eps$, \autoref{thm:main} is stated only for  $\eps\in\bigl\{\frac12,\frac13,\frac14,\ldots\bigr\}$.
The following corollary extends the conclusion to every \(\varepsilon\in(0,1)\).

\begin{corollary}
\label{cor:regular-epsilon}
There exist $c_3,c_4\geq 1$ such that
for every compact metric probability space $(X,d,\mu)$ and every
$\eps\in(0,1)$, there exist a compact subset $S\subseteq X$, an
ultrametric $\rho$ on $S$, and a Borel probability measure $\nu$
supported on $S$ such that,
\[
 d(x,y)\leq\rho(x,y)\leq{64}d(x,y)/{\eps}
 \qquad(x,y\in S),
\]
and, for every $x\in X$ and $r\geq0$,
\[
 \nu(B_d(x,r))
 \leq\mu(B_d(x,{c_3r}/{\eps}))^{1-\eps}.
\]
Moreover, for every $y\in S$ and $r>0$, there is $z\in X$ such that
\[
 B_d(z,{\eps r}/c_4)\subseteq B_d(y,r)
\quad\text{and}\quad
 \nu(B_d(y,r))
 \geq\tfrac14
 \mu(B_d(z,{\eps r/c_4}))^{1-\eps}.
\]
\end{corollary}

\begin{proof}
The case $\diam(X)=0$ is immediate.  Assume $\diam(X)>0$.
Let $t=\lceil\eps^{-1}\rceil$, and set $\alpha=(1-\eps)/(1-1/t)\in(0,1]$.
Apply \autoref{lem:hierarchy} with the parameter $t$ and let $\sigma$ be the flow on
$\Tprime$ supplied by \eqref{eq:hierarchy-flow}.  The function
$\zeta(u)=\sigma(u)^\alpha$ is nonincreasing and subadditive on
$\Tprime$.  A second application of \autoref{lem:balanced} gives a
subtree and a probability flow $\bar\sigma$ satisfying
\[
 \tfrac12\sigma(u)^\alpha<\bar\sigma(u)\leq\sigma(u)^\alpha.
\quad \text{Consequently,} \quad
 \tfrac14\mu(C_u)^{1-\eps}
 <\bar\sigma(u)\leq\mu(E_u)^{1-\eps},
\]
where the lower bound uses $2^{-1-\alpha}\geq1/4$.  Repeating the
geometric part of the proof of \autoref{thm:main} with $\bar\sigma$
gives the stated estimates with $32t$, $129t$, and $5120t$ in place
of $64/\eps$, $258/\eps$, and $10240/\eps$, respectively. Setting $c_3=129$ and $c_4=10240$, the result
follows from $t\leq2/\eps$.
\end{proof}

\printbibliography[title={References}]

\end{document}